\documentclass[12pt]{amsart}
\usepackage{graphicx}
\usepackage{amsfonts}
\usepackage{amssymb}
\usepackage{latexsym}
\usepackage{amscd}

\usepackage{color, soul}
\definecolor{gr}{rgb}{0.7, 1, 0.7}
\definecolor{rr}{rgb}{1, 0.7, 0.7}

\newtheorem{thm}{Theorem}[section]

\newtheorem{prop}[thm]{Proposition}

\newenvironment{pf*}[1]{\proof[#1]}{\endproof}
\usepackage{euscript}

\usepackage[OT2,OT1]{fontenc}

\newcommand{\cal}[1]{{\mathcal #1}}

\theoremstyle{definition}
\newtheorem{defn}{Definition}[section]

\theoremstyle{remark}
\newtheorem{rem}{Remark}[section]

\newcommand{\wtl}{\widetilde}
\newcommand{\eps}{\epsilon}

\newcommand{\aaa}[1]{{{\mathbf{#1}}}}

\numberwithin{equation}{section}

\newcommand{\cA}{{\cal A}}

\newcommand{\cU}{{\mathcal U}}

\newcommand{\cB}{{\aaa B}}

\newcommand{\bT}{{\mathbf T}}

\newcommand{\cP}{{\cal P}}

\newcommand{\cAC}{{\cal A\cal C}}

\newcommand{\cR}{{\cal R}}

\newcommand{\CC}{{\Bbb C}}

\newcommand{\RR}{{\Bbb R}}
\newcommand{\TT}{{\Bbb T}}
\newcommand{\ZZ}{{\Bbb Z}}
\newcommand{\NN}{{\Bbb N}}

\newcommand{\QQ}{{\Bbb Q}}

\begin{document}
\addtolength{\evensidemargin}{-0.7in}
\addtolength{\oddsidemargin}{-0.7in}

\title[Maps with multiple breaks]{Hyperbolicity of renormalization for maps with multiple breaks}
\author{N. Goncharuk, I. Gorbovickis, M. Yampolsky}
\date{\today}

\maketitle

\begin{abstract}
We construct hyperbolic horseshoes for piecewise-analytic  homeomorphisms of the circle with {\it multiple} break-type singularities, under the assumptions of bounded type rotation numbers and bounded geometry -- provided the sizes of the breaks are uniformly small. As a consequence, we obtain a $C^{1+\alpha}$-rigidity result for such maps and prove that rigidity classes are analytic submanifolds.
  \end{abstract}

\section{Introduction}
The subject of this paper are piecewise-smooth homeomorphisms of the circle $\TT=\RR/\ZZ$ whose singularities are of a {\it break} type -- that is, one-sided derivatives exist, are non-zero, but do not match. The magnitude of the break singularity of $f$ at $c\in\TT$ is conveniently estimated as $f'(c+)/f'(c-)$. There is an extensive literature on such maps, primarily in the works of K.~Khanin with co-authors (see e.g. \cite{VK,KhKh,KhanYam}). As is classical from the time of Denjoy \cite{Denjoy}, the principal question in the study of such maps is {\it rigidity}. That is, assuming two maps with irrational rotation numbers are combinatorially equivalent -- which in this case means that there is a topological conjugacy $\phi$ between them which maps breaks to breaks with the same size -- will $\phi$ have to be smooth (and how smooth)?

The modern approach to proving rigidity of circle maps is via the study of their renormalizations. Homeomorphisms with irrational rotation numbers can be renormalized infinitely many times, and the smoothness of the conjugacy between two maps is related to the rate at which the distance between their renormalizations descreases.
Renormalizations of maps with break-type singularities have particuarly nice properties. Specifically, independently of the number of breaks or their size, the sequence of renormalizations of a $C^{2+\text{H\"older}}$ map with breaks with an irrational rotation number converges in $C^1$ metric at a geometric rate to the finite-dimensional space of piecewise M\"obius maps with the given number of breaks. In full generality, this result can be found in \cite{GhKh}, with the estimates carried out in \cite{KhTepl}.
This is in a sharp contrast with maps with critical points for which the corresponding space is infinite-dimensional, leading to a much more challenging renormalization and rigidity theory (see e.g. \cite{GorYa2} and references therein).

For maps with a single break singularity, the renormalization theory is largely completed. In particular, renormalization hyperbolicity on the finite-dimensional space of piecewise-M\"obius maps was established by Khanin and one of the authors in \cite{KhanYam}; and infinite-dimensional treatment on the space of general analytic maps is going to be addressed in our forthcoming work \cite{GonGorYam}. In contrast, the problem has been wide open for maps with more than one break.

In this paper, we construct a complete renormalization picture -- with a hyperbolic horseshoe of renormalization -- for analytic maps with any number of breaks, and prove the corresponding rigidity result for $C^{2+\text{H\"older}}$-smooth maps under two additional conditions. The first one is that the rotation numbers have to be of bounded type. This is a natural restriction when studying renormalization of maps with singularities as it ensures that the renormalization horseshoe is closed. The extension of the renormalization horseshoe picture to analytic maps with general irrational rotation numbers even in the case of a single break is highly non-trivial and will be carried out in our forthcoming work \cite{GonGorYam}. The second condition is that the break sizes have to be uniformly small. The reason for this condition is a perturbative argument in our proof. It is likely that renormalization hyperbolicity remains true for maps with large breaks as well.

To highlight the main ideas of the argument and to keep the presentation clear,  we have tried to keep the exposition streamlined -- so this paper is not meant to serve as an introduction to the subject. As is standard, renormalization of circle maps is presented in the language of pairs of maps which can be seen as non-linear interval exchange transformations of two intervals. Our framework starts with renormalization horseshoe -- rigid rotations -- constructed for analytic maps without breaks in the work \cite{GY} of two of the authors and then shows that it survives under certain types of perturbations. Thus it fits within the philosophy of KAM theory, showing that conjugacy classes of ``nice'' rotations survive under suitable small perturbations within finite codimension submanifolds. In fact, we prove a sharp KAM result showing that rigidity classes are global submanifolds whose codimension corresponds to the number of breaks.

The survival of a renormalization horseshoe  has been used previously in the works of subsets of the authors, see \cite{GorYa} and \cite{KAM-yam}, however, the synthesis in the present paper is novel. The key point, which makes our whole approach possible, is that in our treatment renormalization becomes  a differentiable -- in fact, analytic -- dynamical system on an infinite dimensional Banach manifold of pairs of piecewise-analytic maps. This permits us to use the toolbox of hyperbolic dynamics to study its properties.
Although seen previously for maps of the circle (see e.g. our works \cite{Ya3,GorYa,GY,KAM-yam}), such a construction is new in the setting of maps with breaks, and 
will be useful in the study of other open questions, including in our upcoming work \cite{GonGorYam}.

\section{Almost commuting pairs}
For a point $z_0\in\CC$ and a real number $r>0$, denote $U_r(z_0)=\{z\in\CC\colon |z-z_0|<r\}$; for a set $S\subset \CC$, let $$U_r(S)=\underset{z\in S}{\cup} U_r(z).$$

For $M\in\NN$, let $\bT_M$ denote the irrational numbers of type bounded by $M$ (that is, $M$ is an upper bound for the terms in the continued fraction expansion of the irrational number).

Let us set
$$T_\theta(z)=z+\theta,\text{ and }\alpha(z)=T_{-1}(z)=z-1.$$
Let us fix a real-symmetric open topological disk $W\subset\CC$ with $W\supset [-1,0]$ and denote by $\cA_W$ the Banach space of analytic maps $f\colon W\to\CC$ continuous up to the boundary of $W$, with the uniform norm.

We say that a pair of maps $\zeta=(\eta,\xi)$ is {\it an almost commuting pair} if
\begin{enumerate}
\item $\eta$, $\xi$ are bounded real-symmetric analytic maps in $W$, continuous up to the boundary;
\item $\zeta$ is a non-linear interval exchange transformation:
  \begin{itemize}
  \item $W\supset [\eta(0),\xi(0)]\ni 0$;
    \item $\eta(0)=-1$;
  \item   $\eta$, $\xi$ have diffeomorphic restrictions to $[0,\xi(0)]$, $[\eta(0),0]$ respectively;
    \item $\eta\circ\xi(0)=\xi\circ\eta(0)\in[\eta(0),0].$
  \end{itemize}
\item $[\zeta](z)=\eta\circ \xi(z)-\xi\circ\eta(z)=o(z)$ as $z\to 0$.
\end{enumerate}

We denote the space of such pairs of maps $\zeta$ by $\cAC_W$; as seen in \cite{GaYa} it is a Banach submanifold of the space of pairs of bounded analytic functions in $W$, continuous up to the boundary. Gluing a real neighborhood of $0$ with a real neighborhood of $\eta(0)$ via $\eta$, we obtain a smooth manifold $M$ homeomorphic to the circle $\TT=\RR/\ZZ$. The map $\xi$ induces a well-defined $C^1$-smooth map $f_\zeta$ on $M$. The rotation number of the pair $\zeta$ is given by
$$\rho(\zeta)\equiv \rho(f_\zeta).$$

Renormalization $\cR$ is defined in the usual way. Namely, if $\rho(f_\zeta)\neq 0$, set
$$\chi=\chi(\zeta)=[1/\rho(\zeta)].$$
The {\it pre-renormalization } of $\zeta$ is the pair
$$p\cR(\zeta)=(\xi,\xi^\chi\circ \eta)$$
and the renormalization is the rescaling
$$\cR(\zeta)=(\lambda\circ\xi\circ\lambda^{-1},\lambda\circ\xi^\chi\circ \eta\circ\lambda^{-1})$$
by the homothety $\lambda(x)=-x/\xi(0)$.
Renormalization acts as the Gauss map $G(x)=\{1/x\}$ on rotation numbers.

For $M\in\NN$ we set $$\Lambda_M=\{(\alpha,T_\theta)\mid \theta\in\bT_M\}.$$
This set is evidently invariant under $\cR$.

If an almost commuting pair $\zeta=(\eta,\xi)$ satisfies the commutation condition
$$\eta\circ\xi=\xi\circ\eta$$
where defined, then it is called a {\it commuting pair}. Note that in this case the map $f_\zeta$ obtained as above is analytic.

The renormalization of an almost commuting pair is an almost commuting pair and the renormalization of a commuting pair is again commuting.
Since the
commutation condition has infinite order, it does not readily translate into Banach submanifold structure on for the set of commuting pairs, this is a known difficulty which the almost commutation condition allows one to finesse.

Renormalization hyperbolicity results of \cite{GY} are formulated in the Banach space of analytic maps defined on a neighborhood of the circle using the definition of a {\it cylinder renormalization}. An equivalent formulation would be to consider {\it normalized commuting pairs}
$\zeta=(\alpha,\xi)$ \cite{KAM-yam}. The definition of renormalization would then need to include a canonical choice of a uniformizing conformal coordinate change to turn the first map of the renormalized pair into the unit translation. This is done in \cite{GY} (in much more generality than for bounded-type rotation numbers) and in this setting the set $\Lambda_M$ is shown to be hyperbolic with one-dimensional expanding direction. Moreover, its stable manifolds coincide with conformal conjugacy classes of elements of $\Lambda_M$. We refer the reader to \cite{GY} for further details.

The renormalization of an almost commuting pair is again an almost commuting pair, and 
as seen in \cite{KAM-yam}, renormalization is {\it commutativity improving}. We quote \cite[Theorem 2.2]{KAM-yam}:
\begin{thm}
  \label{th:commutator1}
  Fix  $\tau\in(0,1)$, and let $M\in\NN$. There exist  $\ell\in\NN$ and  $\eps>0$
    such that the following is true. 
  Let $\zeta=(\eta,\xi)\in\cAC_W$ be $\eps$-close to $\Lambda_M$ and $\rho(\zeta)\in \bT_M$. Then
  \begin{equation}
    \label{commutator-estimate1}
    \|[\cR^\ell \zeta]\|^\infty_{U_\delta(0)}<\tau\|[\zeta]\|^\infty_{U_\delta(0)}.
    \end{equation}
\end{thm}
\begin{proof}
  Easy induction shows (cf. \cite[Lemma 2.3]{KAM-yam}) that, denoting by
$p\cR^\ell\zeta=(\eta_\ell,\xi_\ell)$ the $\ell$-th pre-renormalization of $\zeta$, we have
\begin{equation}
  \label{eq-comm-comp}
  \eta_\ell\circ\xi_\ell=f_\ell\circ\eta\circ\xi\text{ and }\xi_\ell\circ\eta_\ell=f_\ell\circ\xi\circ\eta,
  \end{equation}
where $f_\ell$ is  the same composition of iterates of $\eta$ and $\xi$. In particular, the commutator $[p\cR^\ell\zeta]$ has the form
\begin{equation}
  \label{eq-comm}
  [p\cR^\ell\zeta]=f_\ell\circ\eta\circ\xi-f_\ell\circ\xi\circ\eta.
\end{equation}
  In view of the above, we have the following first-order estimate:
  $$[p\cR^\ell\zeta]=\eta_\ell\circ\xi_\ell-\xi_\ell\circ\eta_\ell\sim f_\ell'(\eta\circ\xi(0))[\zeta].$$
Set $\lambda_\ell=|\xi_\ell(0)|$.
 Select $\ell$ large enough, so that $\lambda_\ell \ll \tau$.   Assuming the pair is sufficiently close to $\Lambda_M$, the derivative $f_\ell'$ is close to 1. Since $[\zeta] \sim k z^2$ at zero with $|k|<c \|[\zeta]\|^{\infty}_{U_\delta(0)}$, we get
  $$|[\cR^\ell\zeta](z)|< \lambda_\ell |f_\ell'(\eta\circ\xi(0))|\cdot c \|[\zeta]\|^{\infty}_{U_\delta(0)} < \tau  \|[\zeta]\|^{\infty}_{U_\delta(0)},$$
and the statement follows.
\end{proof}

The renormalization hyperbolicity result of \cite{GY} naturally translates into the setting of almost commuting pairs \cite{KAM-yam} (we only quote it for maps of bounded type, as this is the setting we will be working in). 

\begin{thm}
  \label{th:hyperb1}
Let $M\in\NN$.
There exists a domain $W$ %and an even natural number $N$
such that the following holds. Renormalization operator $\cR^2$
is an analytic operator from an open neighborhood $\wtl\cU\Supset \Lambda_M$ in $\cAC_W$ to a compact subset of $\cAC_W$. Its differential at every point is a compact linear operator. The invariant set $\Lambda_M$  is uniformly hyperbolic under $\cR^2$, with one-dimensional unstable direction.

%It has a codimension one strong stable foliation by analytic submanifolds.

%The convergence to $\Lambda_M$ is global: let $\zeta$ be an almost commuting pair with $\rho(\zeta)\in \bT_M$. Then
%$$\cR^k(\zeta)\to\Lambda_M$$
%in the locally uniform sense.
  \end{thm}

%%  TODO I think it is OK to use the second iterate of renormalization here.

  \begin{proof}
 %% \emph{Comment: We decided to work with differentials only, without constructing stable manifolds, till we get to the end --- and use Hadamard-Perron there. }

The compactness statement follows since $\mathcal R$ increases domain of definition of the maps for pairs that are close to translations.
The complex line $\{(\alpha,T_\theta)\mid \theta\in\CC\}$ is an invariant unstable manifold for all of $\Lambda_M$, since the action of renormalization is the second iterate of the Gauss map $\theta\to 1/\theta$.

A straightforward quasiconformal surgery argument (see \cite[Theorem 7.9]{GorYa2}) shows that the distance from an almost commuting pair $\zeta$ to a commuting pair with the same rotation number is bounded by $C\|[\zeta]\|^\infty_{U_\delta(0)}$ with $C=C(W)$. In view of Theorem~\ref{th:commutator1}, we only need to show that on the set of commuting pairs with rotation numbers in $\bT_M$, there is an exponentially fast convergence to $\Lambda_M$.

Let $\zeta=(\eta,\xi)$ be a commuting pair with $\theta=\rho(\zeta)\in \bT_M$ that is close to $\Lambda_M$. For commuting pairs, the map $f_\zeta$ constructed above is an analytic circle diffeomorphism. By Yoccoz's theorem,
the map $f_\zeta$ is conformally conjugate to the rigid rotation by the angle $\theta$. It follows that $\zeta$ is conjugate to the pair $(\alpha,T_\theta)$ via a conformal map $\Phi$ defined in a neighborhood of $[\eta(0), 0]$.  Setting $\theta_n=\rho(\cR^n\zeta)$ (which is simply the orbit of $\theta$ under the Gauss map), we see that $\cR^n\zeta$ is conjugate to $(\alpha,T_{\theta_n})$ via a conformal conjugacy $\Phi_n$. Here $\Phi_n$ is a rescaled restriction of $\Phi$, hence $\Phi_n$ uniformly converges to the identity map at a geometric rate.

  \end{proof}

\section{Enriched almost commuting pairs}
Fix a nonnegative integer $K$. An {\it enrichment} of an almost commuting pair  $\zeta$ is a list
$$\bar e=(s_1, s_2, \cdots, s_K)$$
of $K$ marked points in $[\eta(0),\xi(0)]$ in a non-decreasing order.
We denote an enriched pair as above by $\zeta^{\bar e} = (\zeta, \bar e)$.
We impose the manifold structure of $\cAC_W\times \RR^K$ on the space of enriched pairs.
The action of renormalization is extended to the space of enriched pairs in the following way. The pre-renormalization is
$$p\cR(\zeta^{\bar e})=p\cR(\zeta)^{\bar a},$$
where $\bar a=(t_1,t_2,\ldots,t_K)$, and  $t_j\in[\eta\circ\xi^{\chi}(0),\xi(0)]$ is the point whose orbit under the first return map to  this interval contains $s_j$. The point $t_j$ is determined uniquely, unless the choice is between one of the two endpoints -- in which case we will place it at $\xi(0)$.
We then set
$$\cR(\zeta^{\bar e})=\cR(\zeta)^{\bar e'},\text{ with }\bar e'=\tau(-\bar a/\xi(0))$$
where $\tau$ is a substitution that reorders the list in a non-decreasing order. Note that $\tau$ is locally constant in $\zeta$ whenever the points $t_j$ are different.

We remark, that renormalization has discontinuities when marked points coincide with an endpoint of the new interval of definition, since under a small perturbation the marked point can jump either to the other endpoint or to the origin. Thus, there are several different analytic branches of the renormalization operator near such enriched pairs. See Section~\ref{sec:general}, where we finesse this explicitly, for further details.

Let us denote $\Lambda_M^K\subset \cAC_W\times\RR^K$ the thus enriched attractor $\Lambda_M$. It is clear that the action of the differential
$D\cR$ at $(\alpha,T_\theta)\in\Lambda_M$ expands the differential of each enriching coordinate by the multiple $-1/\theta$. Hence
$\Lambda_M^K$ is a hyperbolic set in the space of enriched pairs with $K+1$ unstable directions.

\section{Maps with higher order breaks}
We now extend our space once again. Let us say that two real-symmetric analytic maps $f$ and $g$ have a {\it break of order} $r\geq 0$ at a point $x_0\in\RR$ at which both of them are defined if $f(x_0+h)-g(x_0+h) = O(h^r).$

%For an ordered set of points $\bar x =(x_1, \dots, x_n)$, let $i(x)$ be the same set of point re-ordered in a non-decreasing order.

\begin{defn}\label{defn-miltibreak}
  Fixing a domain $W\Supset [-1,0]$ we say that a tuple $\zeta=(\{f_j\}_{j=1}^{K+2}, \{s_j\}_{j=0}^{K+2})$ is a {\it piecewise-analytic almost commuting pair with breaks of order $r\geq 1$} if the following properties hold:
  \begin{enumerate}
  \item Maps $f_j$ are analytic in $W$, continuous up to the boundary, univalent, and real-symmetric;
  \item The list of real points $\{s_j\}\subset W$ includes $s_l=0$  and satisfies the following.
  \begin{itemize}
    \item $f_j([s_{j-1}, s_{j}])\subset W$;
    \item The pair $f_j$, $f_{j+1}$ has a break of order $r$ at $s_{j}$ for $j\neq l$;
    \item The maps $f_{l}$ and $f_{l+1}$ satisfy $f_l(0)=s_{0}$, $f_{l+1}(0)=s_{K+2}$,  and $[f_l, f_{l-1}] = O(z^r)$.

    \end{itemize}

  \end{enumerate}

\end{defn}
Note that the only interesting case will be the case when $f_l(0)=s_0\le s_1\le \dots \le s_{K+1}\le f_{l-1}(0)=s_{K+2}$, in which case the tuple defines an almost commuting pair of continuous mappings $\zeta=(\eta, \xi)$ that equals $f_j$ on $[s_{j-1}, s_{j}]$. However,  we do not require this in the definition to allow perturbations.

We denote the space of such pairs viewed as a subset of $\cA_W^{K+2}\times \RR^{K+2}$
with the induced Banach submanifold  structure by
$\cP_rA_W^K$.
%We say that, in the above notation, there is no break at $s_{j}\neq 0$ if $f_j\equiv f_{j+1}$.
For each fixed $r$,
the space of enriched pairs $\cA_W\times \RR^K$ naturally embeds into $\cP_rA_W^K$, each pair corresponding to a tuple of the form $(\xi, \dots, \xi, \eta,\dots,\eta), (\eta(0), s_1, \dots, 0, \dots,  s_K, \xi(0))$.

Renormalization naturally extends to a neighborhood of any $\zeta^{e} \in \cAC_W\times\RR^K $ in $\cP_rA_W^K$ with non-coincident marked points of $\zeta^e$ and $\cR(\zeta^e)$. Namely, renormalization of $\zeta^e$ can be viewed as a tuple of rescaled compositions of the maps $\zeta|_{[s_k, s_{k+1}]}$. Marked points $\tilde s_j$ of the renormalized map are preimages of $s_j$ under compositions of $\zeta|_{[s_k, s_{k+1}]}$, re-ordered via a substitution $\tau$ in a non-decreasing order. By replacing $\zeta|_{[s_j, s_{j+1}]}$ with $f_j$ and fixing the substitution $\tau$, we get a formula for an analytic, compact renormalization operator in a neighborhood of $\zeta^e$ in $\cP_rA_W^K$.

Since the tuple of compositions used in this construction depends on the order of the preimages of $s_j$ in $[0,\eta(0)]$, the renormalization operator $\cR$ will branch on the preimage of the set  $\bigcup_{i\neq j}\{s_i=s_j\}$.

A simple computation (cf. Theorem \ref{th:commutator1}) shows that for any tuple $$\zeta=(\{f_j\}_{j=1}^{K+2}, \{s_j\}_{j=0}^{K+2}) \in \cP_2A_W^K$$ such that $f_j$ are close to translations, renormalization decreases differences $$d_j(\zeta) = \|f_{j+1}-f_j\|_{D_{\eps}(s_j)}^{\infty}$$ geometrically fast. On the other hand, differences  $\|f_{j+1}-f_j\|_{W}^{\infty}$ remain bounded by some constant $C$. Hadamard's three-circle theorem implies that on a smaller domain $[0,1]\subset W_1\subset W$, differences $\|f_{j+1}-f_j\|_{C(W_1)}$ decrease geometrically fast. Hence, renormalizations of $\zeta$ approach the space $\cAC_W\times\RR^K$ geometrically fast in the metric $C(W_1)$, and since renormalization increases domains of definition, the same holds in the metric $C(W)$. We conclude that no new unstable directions appear for pairs with breaks of order $2$.

More specifically, let $\hat \Lambda_M^K\subset \Lambda_M^K$ denote a set of enriched pairs for which no two marked points lie in the same orbit of the pair. 
  
  \begin{thm}\label{thm-hyperb3}
The enriched interval exchange transformations $\hat \Lambda_M^K$ remain an invariant hyperbolic set for the operator $\cR^2$ in $\cP_2A_W$ with $K+1$ expanding directions.

Namely, there exists $\eps>0$ such that any orbit $\{\zeta_k\}_{k\in \ZZ}$ in $\hat \Lambda_M^K$ is a hyperbolic set of respective analytic branches of  $\cR^2$ defined in $\eps$-neighborhoods of $\zeta_k$ in $\cP_2A_W$,  with $K+1$ unstable directions, and with uniform  estimates on expansion and contraction rates.
    \end{thm}

    \begin{rem}
      We removed the pairs for which several marked points belong to the same orbit to avoid having to deal with branching of $\cR$. However, $\eps$ is uniform and does not depend on the distance to such pairs, since hyperbolicity result applies to respective branches of $\cR$.
      
      Furthermore, for every element  $\zeta\in \Lambda_M^K$, there is a value of $n\in\NN$ such that the renormalization
      $$\cR^{2n}\zeta\in \hat\Lambda_M^{K_1}\text{ with }K_1\leq K,$$
      at which point Theorem~\ref{thm-hyperb3} can be applied to establish the hyperbolicity of its $\cR^2$-orbit.
    \end{rem}
    \section{General maps with multiple breaks}
    \label{sec:general}

  Let us denote the space of pairs with breaks of order $1$ by
  $$\cB_W^K\equiv \cP_1A_W^K.$$

In the notation of Definition~\ref{defn-miltibreak}, for $\zeta\in\cB_W^K$,  let us denote
$\bar\beta(\zeta)$ the tuple
$$\bar\beta(\zeta)_j =\left[ \begin{array}{l}\log \displaystyle\frac{f'_j(s_{j+1})}{f'_{j+1}(s_{j+1})} \text{ for } 1\le j\le K+1, j\neq l \\[12pt]
    \log [f_l, f_{l+1}]'(0) \text{ for } j=l\\
  \end{array}\right.$$
  We call $\bar\beta(\zeta)$ the {\it multibreak parameter} of $\zeta$. The last number is called the break at zero. Note that breaks are preserved in renormalization:
   \begin{prop}\label{prop-multibreak}
    In the above notation,
    $$\bar\beta(\cR\zeta)=-\hat \tau(\bar\beta(\zeta))$$
    where $\hat \tau\in S_{K+1}$ is a substitution of $\{1, 2,\dots, K+1\}$ defined by $\zeta$ and locally constant in $\zeta$.
    \end{prop}
The substitution $\hat \tau$  re-orders the list of nonzero marked points $s_j, j\neq l$ in the same way as the substitution $\tau$ introduced above, and also accounts for changing  position of $s_l=0$ in the list $\{s_j\}$.

%Let $R_{\zeta}(\bar \beta) = -\theta(\bar\beta)$ where $\theta$ is as above.
We denote $\cB_W^{K, \bar\beta}$ the Banach submanifold of pairs with a fixed multibreak parameter.
On a neighborhood of $\zeta$,
$$\cR (\cB_W^{K,\bar\beta}) = \cB_W^{K,-\hat \tau(\bar\beta)}.$$
Hence (any analytic branch of) $\cR^2$ can be viewed as an analytic operator in the disjoint union of $(K+1)!$ Banach spaces of the form $\cB_W^{K, \hat \tau(\bar\beta)}$ for all $\hat \tau\in S_{K+1}$.
By Theorem \ref{thm-hyperb3},  $\cR^2$ has a hyperbolic set in $ \cup_{\hat \tau\in S_{K+1}} \cB_W^{K,0}$ of the form $\cup_{\hat \tau\in S_{K+1}} \hat \Lambda_{M}^K$.

One can construct a one-to-one analytic projection $\pi\colon \cB_W^{K,\bar\beta}\to \cB_W^{K,0}$, e.g. by post-composing each $f_j$ with a suitable affine map to remove a break at $s_j$. Hence for small  $\bar \beta$,  branches of  $\cR^2|_{\cB_W^{K,\bar\beta}}$ can be viewed as small perturbations of branches of $\cR^2|_{\cB_W^{K,0}}$.
Structural stability of hyperbolicity implies the following:
\begin{thm}\label{thm-hyperb4}
  There exists $\delta>0$ such that the following holds. For $\|\bar\beta\|_1\leq\delta$, there exists a continuous in $\bar\beta$ family of $\cR^2$-invariant sets $\cup_{\hat \tau\in S_{K+1}}\hat\Lambda_M^{K,\hat \tau(\bar\beta)}$ with $\hat\Lambda_M^{K, \bar 0} = \hat \Lambda_M^{K}$ that are hyperbolic for $\cR^2$.

  Namely, for a universal constant $\eps>0$, each orbit $\{\zeta_k\}_{k\in \ZZ}\subset \hat \Lambda_M^{K, 0}$ embeds into a uniformly continuous family of orbits $\{\zeta_k^{\bar\beta}\}_{k\in \ZZ}$, each being hyperbolic in its $\eps$-neighborhood under the action of respective branches of $\cR^2$, with $K+1$ expanding directions.

%  The set $\Lambda_M^{K, \bar \beta}$ serves as a hyperbolic attractor for pairs $\zeta\in\cB_W^{K,\bar\beta}$ with rotation numbers $\rho(\zeta)\in\bT_M$.
      \end{thm}
%  \begin{thm}\label{thm-hyperb4}
%  For every $\zeta\in \Lambda_M^{K}$, there exists $\eps>0$ such that the following holds. For $||\bar\beta||_1<\eps$, there exists a continuous  family $\zeta^{\bar \beta}$ with multibreak parameter $\bar \beta$ and rotation number $\rho(\zeta^{\bar\beta}) = \rho(\zeta)$,  whose orbit is hyperbolic under $\cR$ with $K+1$ expanding directions in the disjoint union $\bigcup \cB_W^{K, R^n(\bar\beta)}$.XS  \end{thm}

    As shown in \cite{KhTepl}, the sequence of renormalizations for maps with multiple breaks converges to a sequence of piecewise-Moebius maps geometrically fast in $C^1(\RR)$.
    Applying this statement to tuples from $\hat \Lambda_M^{K, \bar \beta}$, we get that  $\hat \Lambda_M^{K, \bar \beta}$ consists of tuples of piecewise-Moebius maps.

Recall, that each $\zeta\in \cB_W^{K,\bar \beta}$ with points $s_j$ in a non-decreasing order defines a circle map with breaks. If this map has irrational rotation number of bounded type, it is continuously conjugate to an irrational rotation, via a map $\phi$ that is unique up to translation. Normalizing this conjugacy by $\phi(0)=0$, we will say that \emph{positions} of the breaks for $\zeta$ are values of $\phi$ at the marked points $s_j$, $1\le j\le K+1$, $j\neq l$.

  Furthermore, we obtain the following  statement:
  \begin{thm}\label{thm-stable}
    Let $\zeta^{\bar \beta}\in\hat \Lambda_M^{K,\bar\beta}$ be a continuous family of pairs as above. Then for small $\|\bar \beta\|$, the stable manifold of $\zeta^{\bar \beta}$ in $\cB_W^{K,\bar \beta}$  is a local analytic manifold of codimension $K+1$.

    The rotation number and the positions of the breaks remain constant along the stable manifold and do not depend on $\bar \beta$.
%%      \item let $\gamma\in W$. Then $\gamma$ is $C^{1+\alpha}$ conjugate to $\zeta$ with $\alpha=\alpha(M)$;
   %     \item reducing $\eps$ in Theorem~\ref{thm-hyperb4} if needed, we have a global rigidity result: $W$ is globally an analyic submanifold, and every $\gamma\in \cB_W^{K,\beta}$ which is continuously conjugate to $\zeta$ by a conjugacy which maps break to breaks of equal size lies in $W$. In other words, $W$ is equal to the combinatorial class of $\zeta$ in the space of analytic pairs.
%      \end{enumerate}
    \end{thm}
  \begin{proof}
The first statement follows from the infinite-dimensional Hadamard-Perron Theorem, see \cite{elbialy1}. The second statement follows from the fact that subsequent renormalizations of pairs from the stable manifold remain in a small neighborhood of $\cR^k(\zeta^0)$.
  %  The global version of rigidity formulated in part (2) is a direct consequence of compactness of the image of the renormalization operator.

  \end{proof}

  We conclude with a global rigidity result, for which we will need to further restrict the combinatorics of the pairs.

  \begin{defn}
Let us say that a pair $\zeta$ with multiple breaks is of  {\it $\delta$-separated type } for $\delta>0$ if $\rho(\zeta)\in\RR\setminus\QQ$ and, for all $n\geq 0$, the positions of the marked points of $\cR^{2n}\zeta$ differ pairwise by at least $\delta$. 
    \end{defn}

  \noindent
  We will simply say that $\zeta$ is of a separated type if such a $\delta$ can be found for it.
  For instance, periodic orbits of $\cR^2$ are all of a separated type. 

  \begin{thm}\label{thm-rigidity}
    Suppose that $\zeta^{\bar \beta}\in \hat \Lambda_M^{K,\bar \beta}$  is a continuous family of pairs as above. Assume that
    $\zeta^0$ is of a $\delta$-separated type.
    %%Suppose that for some $\delta$, all marked points of $\mathcal R^n \zeta^0$ for $n>0$ remain at a distance at least $\delta$ from each other.
    Then for $\|\bar \beta\|<C(\delta)$, the following holds.

Let $\kappa>0$ and let $\gamma$ be a $C^{2+\kappa}$-smooth pair with multiple breaks (cf. \cite{GhKh}). Suppose that there exists a continuous conjugacy $\phi$ between $\gamma$ and $\zeta^{\bar\beta}$ which maps breaks to breaks of the same size. Then $\phi\in C^{1+\alpha}$ for some $\alpha>0$.
  Thus combinatorial equivalence implies $C^{1+\alpha}$ conjugacy for $C^r$-smooth multicritical pairs with $r>2$.
  \end{thm}

  It is an interesting question whether the condition of a separated type can be weakened or even abandoned. It is well-known that degeneration of geometry due to the rotation number of an unbounded type can lead to a loss of rigidity in the smooth case for circle maps with singularities (see e.g. \cite{AvilaRigidity}) -- but the mechanism by which a degeneration appears is different here.

  \begin{proof}[Proof of Theorem~\ref{thm-rigidity}]
    Due to \cite{KhTepl}, the sequence of renormalizations of $\gamma$ converges to a sequence of piecewise-Moebius maps $\gamma_k$ geometrically fast in $C^1(\RR)$. Since $\gamma$ and $\zeta^{\bar\beta}$ are continuously conjugate, distance between break points of   $\gamma_k$ remains bounded below.   Moreover, due to \cite[Proposition 4]{GhKh}, second derivatives of these Moebius maps on respective intervals between breaks remain bounded. Hence respective Moebius maps belong to a compact set in the space of Moebius maps. We conclude that for any $\eps>0$ there exists  $C(\eps, \delta)$ such that for  $\|\beta\|<C(\eps, \delta)$, Moebius maps that define $\gamma_k$ are $\eps$-close to translations and are well-defined in $W$.

    Therefore tuples $\cR^{2k}(\gamma)$  are exponentially close in $C^1(\RR)$ to tuples of Moebius maps $\zeta^{\bar \beta}_{2k} \in \cB_{W}^{K,\hat \tau(\beta)}$ that stay in a small neighborhood of $ \cR^{2k} (\zeta^{0})$. The latter form a pseudo-orbit in the uniform metric in $W$:  Moebius maps of the tuple $\cR^2(\zeta^{\bar \beta}_{2k})$ are exponentially close to Moebius maps of $\zeta^{\bar \beta}_{2k+2}$ in $C^1(\RR)$, hence in $C(W)$.
%    For arbitrary $\gamma$,   on intervals between breaks that are long enough in comparison with $\beta$, restrictions of  $g_k$ are Moebius maps that are well-defined in $W$. If some of the intervals are small in comparison with $\beta$, we will modify $g_k$ to keep the same size of the breaks and guarantee that $g_k$ is affine on these intervals. We find a sequence $\tilde g_k$ of Moebius maps defined in $W$ such that $\tilde g_k$ is a $c|\beta|$-pseudoorbit for $\cR$. Shadowing lemma implies that some orbit $\hat g_k$ of $\cR$ is $c|\beta|$-close to $\tilde g_k$ in $W$, hence $g_k$.  eventually stays in $c|\beta|$-neighborhood of $\Lambda_M^{K, \bat \beta}$, hence  restrict to Moebius maps that are

Due to Theorem \ref{thm-hyperb4}, this implies that $\zeta^{\bar \beta}_{2k}$ are exponentially close to the points $\cR^{2k}(\zeta^{\bar \beta})$ of the horseshoe, hence $C^1$-distance between $\cR^{2k}(\gamma)$ and $\cR^{2k}(\zeta^{\bar \beta})$ decreases geometrically fast.

    The $C^{1+\alpha}$ smoothness of the conjugacy is a standard consequence of such convergence for circle maps of bounded type. In \cite[Sec.9]{KhKh}, all neccessary estimates are carried out for maps with a single break and rotation number of bounded type. Since only convergence of renormalizations and Denjoy-type estimates were used in \cite{KhKh}, the proof applies \emph{verbatim} to the case of maps with additional break discontinuities.

    \end{proof}

  %We see that:
  %\begin{prop}\label{prop-multibreak}
   % In the above notation,
   % $$\bar\beta(\cR\zeta)=-\bar\beta(\zeta).$$
   % \end{prop}
%As a consequence, the second iterate of $\cR$ fixes the multibreak parameter. Let us call $\cB_W^{K,\bar\beta}$ the corresponding $\cR^2$-invariant Banach subspace of $\cB_W^K$; so $\cPA_W^K=\cB_W^{K,0}$. 

\bibliographystyle{amsalpha}
\bibliography{biblio}
\end{document}